\documentclass[UTF-8,reqno]{amsart}
\usepackage{enumerate}
\usepackage{mhequ}
\usepackage{amssymb,url,color, booktabs,nccmath}
\usepackage[left=3.2cm,right=3.2cm,top=4cm,bottom=4cm]{geometry}
\usepackage{mathrsfs}
\usepackage{enumitem,dsfont}

\usepackage{color}
\usepackage[colorlinks=true]{hyperref}
\hypersetup{
	linkcolor=blue,          
	citecolor=red,        
	filecolor=blue,      
	urlcolor=cyan
}

\definecolor{darkergreen}{rgb}{0.0, 0.5, 0.0}

\numberwithin{equation}{section}

\newcommand{\be}{\begin{eqnarray}}
	\newcommand{\ee}{\end{eqnarray}}
\newcommand{\ce}{\begin{eqnarray*}}
	\newcommand{\de}{\end{eqnarray*}}
\newtheorem{theorem}{Theorem}[section]
\newtheorem{lemma}[theorem]{Lemma}
\newtheorem{proposition}[theorem]{Proposition}
\newtheorem{Examples}[theorem]{Example}
\newtheorem{corollary}[theorem]{Corollary}

\newtheorem{definition}[theorem]{Definition}
\theoremstyle{definition}
\newtheorem{remark}[theorem]{Remark}

\def\eps{\varepsilon}

\def\p{\partial}

\def\<{{\langle}}
\def\>{{\rangle}}
\def\({{\Big(}}
\def\){{\Big)}}

\def\bx{{\mathbf{x}}}

\def\dif{{\mathord{{\rm d}}}}

\def\no{\nonumber}
\def\={&\!\!=\!\!&}

\def\mR{{\mathbb R}}
\def\mS{{\mathbb S}}
\def\mT{{\mathbb T}}

\def\1{{\mathbf{1}}}

\def\E{\mathbf E}

\def\geq{\geqslant}
\def\leq{\leqslant}

\def\div{\mathord{{\rm div}}}

\def\eps{\varepsilon}

\def\p{\partial}

\def\<{{\langle}}
\def\>{{\rangle}}
\def\({{\Big(}}
\def\){{\Big)}}

\def\bx{{\mathbf{x}}}

\def\dif{{\mathord{{\rm d}}}}

\def\no{\nonumber}
\def\={&\!\!=\!\!&}
\def\bt{\begin{theorem}}
	\def\et{\end{theorem}}
\def\bl{\begin{lemma}}
	\def\el{\end{lemma}}
\def\br{\begin{remark}}
	\def\er{\end{remark}}
\def\bx{\begin{Examples}}
	\def\ex{\end{Examples}}
\def\bd{\begin{definition}}
	\def\ed{\end{definition}}
\def\bp{\begin{proposition}}
	\def\ep{\end{proposition}}
\def\bc{\begin{corollary}}
	\def\ec{\end{corollary}}

\def\geq{\geqslant}
\def\leq{\leqslant}

\def\div{\mathord{{\rm div}}}

\def\<{\langle} \def\>{\rangle}

\allowdisplaybreaks

\begin{document}
\title{Kolmogorov {$4/5$} law for the forced 3D Navier--Stokes equations}

\author{Martina Hofmanov\'a}
\address[M. Hofmanov\'a]{Fakult\"at f\"ur Mathematik, Universit\"at Bielefeld, D-33501 Bielefeld, Germany}
\email{hofmanova@math.uni-bielefeld.de}

\author{Umberto Pappalettera}
\address[U. Pappalettera]{Fakult\"at f\"ur Mathematik, Universit\"at Bielefeld, D-33501 Bielefeld, Germany}
\email{upappale@math.uni-bielefeld.de}

\author{Rongchan Zhu}
\address[R. Zhu]{Department of Mathematics, Beijing Institute of Technology, Beijing 100081, China}
\email{zhurongchan@126.com}

\author{Xiangchan Zhu}
\address[X. Zhu]{ Academy of Mathematics and Systems Science,
Chinese Academy of Sciences, Beijing 100190, China}
\email{zhuxiangchan@126.com}

\thanks{
R.Z. and X.Z. are grateful for
the financial supports   by National Key R\&D Program of China (No. 2022YFA1006300).
R.Z. gratefully acknowledges financial support from the NSFC (No.
11922103, 12271030).
  X.Z. is grateful for the financial supports  in part by National Key R\&D Program of China (No. 2020YFA0712700) and the NSFC (No.   12090014, 12288201) and
  the support by key Lab of Random Complex Structures and Data Science,
 Youth Innovation Promotion Association (2020003), Chinese Academy of Science.
The research of  M.H. and U.P. was funded by the European Research Council (ERC) under the European Union's Horizon 2020 research and innovation programme (grant agreement No. 949981). The research of M.H., R.Z. and X.Z.  was funded by the Deutsche Forschungsgemeinschaft (DFG, German Research Foundation) -- SFB 1283/2 2021 -- 317210226.
}

\begin{abstract}
We prove that the solutions to the 3D forced Navier--Stokes equations constructed  by Bru\`e, Colombo, Crippa, De~Lellis, Sorella in \cite{BCCDLS22} satisfy an $L^p$-in-time version of the celebrated  Kolmogorov  4/5 law for behavior of the  averaged third order longitudinal structure function along the vanishing viscosity limit. {The result has a natural probabilistic interpretation: the predicted behavior is observed on average after waiting for some sufficiently generic random time.} This is then applied to derive a bound for the exponent of the third order absolute structure function in accordance with the Kolmogorov turbulence theory.
 \end{abstract}

\date{\today}

\maketitle

\tableofcontents
	\section{Introduction}

The celebrated Kolmogorov 1941 theory of turbulence \cite{Kol41a, Kol41b, Kol41c} (see also \cite{Fr95,BV19} for reviews) predicts various universal behaviors of statistics of fully developed turbulence. On the mathematical side, the problem in the context of Navier--Stokes equations  is very delicate and the rigorous understanding of these predictions  is still in its infancy.
Recently, partial results were obtained in \cite{BDL23,BCCDLS22} for  the forced Navier--Stokes system on $[0,1]\times \mT^3$ for incompressible fluids
	\begin{align}
	\label{eq:v}
		\aligned
		\p_t v_\nu+v_\nu\cdot\nabla v_\nu+\nabla p_\nu&=\nu \Delta v_\nu+f_\nu,
		\qquad\div v_\nu=0,\\
		v_{\nu}(0)&=v_{in}.
		\endaligned
	\end{align}
Here $v_{\nu}:[0,1]\times\mT^{3}\to \mR^{3}$ is the velocity field, $p_\nu:[0,1]\times\mT^{3}\to \mR$ the associated pressure, $\nu>0$ is the  kinematic viscosity of the fluid and $f_\nu: [0,1]\times \mT^3\to\mR^3$ is an external force.

The works \cite{BDL23,BCCDLS22} successfully constructed forces $f_{\nu}$ and initial data $v_{in}$ so that the associated solutions $v_{\nu}$ to \eqref{eq:v} satisfy the anomalous dissipation of energy, i.e. solutions $v_{\nu}$ do not preserve their kinetic energy at time $t=1$ in the vanishing viscosity limit.
This result can be seen as a  fixed-time version of the so-called zeroth law of turbulence, which is one of the basic predictions of the  Kolmogorov turbulence theory.
The form of the anomalous dissipation established in  \cite{BDL23,BCCDLS22} 
 reads as
\begin{equation}\label{eq:h}
\varepsilon:=\limsup_{\nu\to 0}\nu \int_{0}^{1}\|\nabla v_{\nu}\|_{L^{2}}^{2}\dif  t > 0,
\end{equation}
and some variant of this result is usually necessary in order to derive further laws of the Kolmogorov theory, in particular the  $4/5$ law and the closely related $4/3$ law, see \cite{Fr95, BCZPSW19}.

The $4/5$ law is an exact relation for the third order longitudinal structure function, i.e., the average of the  cube of the longitudinal velocity increment
given for $\ell>0$, $t\in[0,1]$ by
\begin{align}\label{eq:s1}
	S_{\|}^\nu(t,\ell)&:=\frac1{4\pi}\int_{\mS^2}\int_{\mT^3}(\delta_{\ell\hat{n}} v_\nu(t,x)\cdot \hat{n})^3\dif x\dif S(\hat{n}),
\end{align}
where
$$
\delta_h v_\nu(x)=v_\nu(x+h)-v_\nu(x),\quad h\in\mR^3.
$$
The $4/3$ law then describes the behavior of the averaged structure function
\begin{align}\label{eq:s2}
	S_0^\nu(t,\ell)&:=\frac1{4\pi}\int_{\mS^2}\int_{\mT^3}|\delta_{\ell\hat{n}} v_\nu(t,x)|^2\delta_{\ell\hat{n}} v_\nu(t,x)\cdot \hat{n}\dif x\dif S(\hat n).
\end{align}
Heuristically, the laws read respectively  as
$$
\int_0^1 S_{\|}^\nu(t,\ell) \dif t\thicksim -\frac45 \varepsilon\ell\qquad \mbox{and}\qquad
\int_0^1 S_{0}^\nu(t,\ell) \dif t \thicksim -\frac43 \varepsilon\ell
$$
over a range of scales $\ell\in[\ell_D,\ell_I]$.
The interval $[\ell_D,\ell_I]$ is commonly called the inertial range. The top scale $\ell_I$, called the integral scale, is essentially the scale at which the external forcing acts and injects energy while $\ell_D$, called the dissipation scale, is the scale where the viscosity dominates the flow and the energy is dissipated.

The Kolmogorov theory of turbulence was developed under the assumptions of homogeneity, isotropy as well as self-similarity. 
 The assumption of homogeneity and isotropy can be conveniently replaced by averaging over the spatial variable $x$ and the direction $\hat{n}$ as we did in \eqref{eq:s1}, \eqref{eq:s2}. Self-similarity is not supported by experimental studies due to what is now commonly referred to as intermittency (see \cite{BV19} for a review). Nevertheless, the Kolmogorov $4/5$ law matches very well with experiments. In fact, it is regarded as an exact law of turbulence in the physics community. However,  to the best of our knowledge, up to now  there are no known examples where this law was mathematically rigorously proved in the framework of the Navier--Stokes equations. The validity of alike universal laws for any simplified model of turbulence is a strong physical desideratum, which can be of help in ruling out non physical and artificial phenomenological models.

We  mention that, under a weak anomalous dissipation assumption \cite{BCZPSW19} proved both the $4/5$ and the $4/3$ law for  statistically stationary martingale solutions to the 3D  Navier--Stokes equations with an additive noise. But it remains an open question whether such statistically stationary solutions exist (see \cite[Remark 1.6]{BCZPSW19}). We also refer the readers to  \cite{NT99} which considers deterministic, smooth solutions with a stronger hypothesis.

In this paper, we show that the solutions constructed in \cite{BCCDLS22} satisfy an $L^{p}$-version of both these laws, supporting further their relevance to the study of turbulence. Our proof is inspired by the presentation in \cite{BCZPSW19}. In the first step, we derive the so-called K\'arm\'an--Howarth--Monin relation, which for a particular choice of test functions reduces to expressions for the structure functions $S^{\nu}_{\|}$ and $S^{\nu}_{0}$. The main difference to \cite{BCZPSW19} now lies in the fact that we cannot use stationarity to cancel out certain terms and  need to control them.
 This requires some uniform in $\nu$ regularity of the solutions $v_{\nu}$ and a uniform in $\nu$ integrability of the forces $f_{\nu}$.

To be more precise, we work under the general assumption

(\textbf{H}) Assume that $v_\nu$ is a Leray--Hopf solution to \eqref{eq:v} on $[0,1]\times \mT^3$
 such that there exist $\sigma>0$, $\alpha\in(0,1)$ and $q\in[1,2]$ such that $$\sup_{\nu\in(0,1)}\Big(\|v_\nu\|_{L^\infty(0,1;L^2)}+\|v_\nu\|_{L^{q}(0,1;H^\alpha)}+\| f_\nu\|_{ L^{1+\sigma}(0,1;L^2)}\Big)<\infty.$$
Here the upper bound $1$ for $\nu$ could be changed to another small fixed constant.
\noindent We recall that Leray--Hopf solutions are weak solutions satisfying an  energy inequality, see Definition~\ref{def:sol}. With this at hand, we obtain the following $L^{p}$-version of the Kolmogorov $4/5$ and $4/3$ law, {which is established via Theorem~\ref{th:34} and Theorem~\ref{45}.} As a matter of fact, even a stronger uniform estimate than (\textbf{H}) holds true for the solutions constructed in \cite{BCCDLS22}, namely they possess the Onsager critical regularity $L^{3}(0,1;C^{1/3-})$, but this is not needed in the proof of the result.

\bt\label{th:main}
Let $v_{\nu}$, $\nu\in (0,1)$, be  solutions to the forced Navier--Stokes system \eqref{eq:v} satisfying the hypothesis {\em (\textbf{H})}. There exists $\ell_D=\nu^K$ with $K<\frac1{2-\alpha q}$ such that  for any $p\in[1,\infty)$
\begin{align*}
	\lim_{\ell_I\to0}\limsup_{\nu\to0}\sup_{\ell\in [\ell_D,\ell_I]}\left\|\int_0^\cdot\frac{S_{\|}^\nu(r,\ell)}{\ell}\dif r+\frac45\eps_\nu(\cdot)\right\|_{L^p(0,1)}=0,
\end{align*}
and
\begin{align*}
	\lim_{\ell_I\to0}\limsup_{\nu\to0}\sup_{\ell\in [\ell_D,\ell_I]}\left\|\int_0^\cdot\frac{S_0^\nu(r,\ell)}{\ell}\dif r+\frac43\eps_\nu(\cdot)\right\|_{L^p(0,1)}=0,
\end{align*}
where
\begin{align}\label{eq:eps}
\varepsilon_{\nu}(t) :=\frac12\|v_\nu(0)\|_{L^2}^2-\frac12\|v_\nu(t)\|_{L^2}^2
+\int_0^t\<f_\nu,v_\nu\>\dif r.
\end{align}
If additionally the energy equality holds true then
$$
\varepsilon_{\nu}(t) = \nu  \int_{0}^{t}\|\nabla v_{\nu}(s)\|_{L^{2}}^{2}\dif s.
$$

\et

\br (i) When $q=2$, $\alpha=1/3-$ (the solutions constructed in \cite{BCCDLS22} satisfy this) then $K=3/4-$ which nearly coincides with the usual K$41$ dissipative length scale $\ell_D(\nu)\sim \nu^{3/4}$. 

(ii) In the first version of this paper, we proved our result in the case $q=1$ only, with suboptimal dissipative length scale $\ell_D\sim \nu^{1/2-}$. After our paper first appeared on arXiv, the author in \cite{No23} improved our dissipative length scale by establishing a number of new energy balance identities. Inspired by his work and the work \cite{GK23}, we also improve the dissipative length scale by a modification of our original proof which follows the classical K\'{a}rm\'{a}n--Howarth--Monin relation as in \cite{BCZPSW19}. We stress here that the regularity condition is also related to the asymptotic behavior of the absolute structure function (see Theorem \ref{thm:main2} below) and Onsager's conjecture.
	
(iii) For the stochastic case in \cite{BCZPSW19}, if we replace the weak anomalous dissipation condition there by the following regularity condition,
	$$\sup_{\nu\in(0,1)}\E\|v_\nu\|_{H^{1/3-}}^2<\infty,$$
	then the Kolmogorov $4/5$ law holds with the dissipative length scale $\ell_D\sim \nu^{3/4-}$.
\er

As the anomalous dissipation is obtained in the time averaged sense, cf. \eqref{eq:h}, it would be natural to average also in time for the $4/5$ and the $4/3$ law, i.e. to integrate $S^{\nu}_{\|}(\cdot,\ell)$ and $S^{\nu}_{0}(\cdot,\ell)$ with respect to $t\in [0,1]$. However, we do not obtain the desired Kolmogorov laws in this case, because the solutions from \cite{BCCDLS22} may   not have a uniform in $\nu$ regularity at time $1$, which would be needed in the proof.   By considering  the $L^{p}$-version of Kolmogorov's $4/5$ and $4/3$ laws instead, we overcome this problem and a uniform in $\nu$ regularity in $L^{1}$ with respect to time as formulated in the hypothesis (\textbf{H}) is sufficient.

{ Furthermore, the result of Theorem~\ref{th:main} has an immediate probabilistic interpretation:
 the predicted behavior is captured  on average after waiting for some random time (not necessarily large) under the condition that this random time is sufficiently ``spread'' across time interval $[0,1]$. Precisely, the result is formulated as follows  {and the proof
is discussed at the end of Section~\ref{s:2}.}}

\begin{theorem}\label{th:1.2}
Let $v_{\nu}$, $\nu\in (0,1)$, be  solutions to the forced Navier--Stokes system \eqref{eq:v} satisfying the hypothesis {\em (\textbf{H})}. There exists $\ell_D=\nu^K$ with $K<\frac1{2-\alpha q}$ such that  for every $p\in[1,\infty)$, $\kappa>0$ and $K>0$
\begin{align*}
	\lim_{\ell_I\to0}\limsup_{\nu\to0}\sup_{\ell\in [\ell_D,\ell_I]} \sup_{\mathfrak{t}}\left\langle\left|\int_0^\mathfrak{t}\frac{S_{\|}^\nu(r,\ell)}{\ell}\dif r+\frac45\eps_\nu(\mathfrak{t})\right|^{p}\right\rangle=0,
\end{align*}
and
\begin{align*}
	\lim_{\ell_I\to0}\limsup_{\nu\to0}\sup_{\ell\in [\ell_D,\ell_I]} \sup_{\mathfrak{t}}\left\langle\left|\int_0^\mathfrak{t}\frac{S_0^\nu(r,\ell)}{\ell}\dif r+\frac43\eps_\nu(\mathfrak{t})\right|^{p}\right\rangle=0.
\end{align*}
Here
$
\varepsilon_{\nu}$ was defined in \eqref{eq:eps}, $\mathfrak{t}$ are arbitrary random times taking values in $[0,1]$ whose law is absolutely continuous with respect to the Lebesgue measure with a density $\psi$ satisfying $\|\psi\|_{L^{1+\kappa}(0,1)}\leq K$ and the bracket $\langle\cdot\rangle$ denotes the ensemble average.
\end{theorem}

Finally, in Theorem~\ref{th:341} we show a modified version of the Kolmogorov $4/3$ law which holds true for the time average of the structure function $S^{\nu}_{0}$ but with the  rate of energy dissipation replaced by  a different quantity. A similar result in the case of the $4/5$ law is proved in Theorem \ref{451}. Theorem~\ref{th:341} is then combined with the Onsager critical regularity of the solutions from \cite{BCCDLS22} to derive an estimate for the exponent of the third order absolute structure function averaged in time as predicted by Kolmogorov's theory. The $p$th order absolute structure function is defined
for any $\ell>0$ and $t\in[0,1]$ as
\begin{align*}
	S_p^\nu(t,\ell)&:=\frac1{4\pi}\int_{\mS^2}\int_{\mT^3}|\delta_{\ell\hat{n}} v_\nu(t)|^p\dif x\dif S(\hat n).
\end{align*}
The Kolmogorov theory predicts the asymptotic behavior
$$S_p^\nu(t,\ell)\thicksim (\eps\ell)^{p/3}$$ for $\ell$ in the inertial range. For $p=3$ this prediction is indeed supported by all the experimental evidence. For $p \neq3$ the experiments  deviate from the Kolmogorov prediction due to intermittency (see \cite{BV19}).

As our last result we  prove that  the solutions obtained in \cite{BCCDLS22} satisfy
$$c\ell\leq \int_{0}^{1}S_3^\nu(t,\ell)\dif t\leq C\ell^{3\alpha},$$
for some $c,C>0$ and $0<\alpha<1/3$, $\ell\in [\ell_D,\ell_I]$.
More precisely, the following result holds and is established in Theorem~\ref{thm:3.1}.

 \bt \label{thm:main2}
Let $\alpha\in (0,1/3)$ be given and let $v_\nu$, $\nu\in(0,1)$, be the  solutions to the forced Navier--Stokes equations \eqref{eq:v} from \cite{BCCDLS22} satisfying
$$
	\sup_{\nu\in (0,1)}\Big(\|v_\nu\|_{L^3(0,1;C^\alpha)}+\|v_\nu\|_{L^\infty(0,1;L^\infty)}\Big)<\infty.
	$$
 Then
  the third order absolute structure function exponents
$$
\zeta_3:=\liminf_{\ell_I\to0}\inf_{\nu\in(0,1)}\inf_{\ell\in[\ell_D,\ell_I]}\frac{\log (\int_0^1S_3^{\nu}(r,\ell)\dif r)}{\log \ell},\qquad\bar\zeta_3:=\limsup_{\ell_I\to0}\liminf_{\nu\to0}\sup_{\ell\in [\ell_D,\ell_I]}\frac{\log (\int_0^1S_3^{\nu}(r,\ell)\dif r)}{\log \ell}.
$$
satisfy
$$
3\alpha\leq{\zeta}_3\leq\bar{\zeta}_3\leq 1.
$$
\et

We  note that the result in Theorem~\ref{thm:main2} holds for the absolute structure function averaged in time which is the natural way to treat nonstationary solutions as discussed above.

\br In fact, Kolmogorov developed  his turbulence theory in the statistically stationary regime, or more generally after taking suitable ensemble/long-time averages. This averaging is a point where ambiguities often arise. Among other reasons,  in experiments the long-time averages are usually taken  as a replacement for ensemble averages, which implicitly assumes the validity of the so-called ergodic hypothesis. In the setting of weak solutions without the energy inequality, i.e., not Leray--Hopf solutions, the ergodic invariant measures are not unique as shown in \cite{HZZ22}. Up to date, it is unknown, whether the energy inequality could select only one unique ergodic invariant measure, and other selection criteria are not available. As a consequence, any attempt of establishing rigorously the validity of the Kolmogorov 1941 theory at stationarity, without ruling out this ambiguity, might well be a hopeless task.
\er

	\section{Kolmogorov $4/5$ and $4/3$ law}\label{s:2}
	We first recall the notion of Leray--Hopf solution to the forced Navier--Stokes equations \eqref{eq:v}.

\bd \label{def:sol} Let $v_{\text{in}}\in L^2(\mT^3)$ be a divergence-free vector field, and $f_{\nu}\in L^1(0,1;L^2)$. A Leray--Hopf solution to the Navier--Stokes system \eqref{eq:v} on $[0,1]\times \mT^3$ with initial data $v_{\text{in}}$ and force $f_\nu$ is a divergence-free vector field $v_\nu\in L^\infty(0,1;L^2)\cap L^2(0,1;H^1)\cap C_w([0,1];L^2)$ such that $v_\nu(0)=v_{\text{in}}$ and for all $t\in [0,1]$ and all $\psi\in C^\infty(\mT^3)$
\begin{align*}
	\<v_\nu(t),\psi\>-\<v_\nu(0),\psi\>=\nu\int_0^t\<v_\nu,\Delta\psi\>\dif r-\int_0^t \<v_\nu\cdot \nabla v_\nu,\psi\>\dif r+\int_0^t \<f_\nu,\psi\>\dif r-\int_0^t\<\nabla p_\nu,\psi\>\dif r,
\end{align*}
  and the  following energy inequality holds true for all $t\in(0,1]$
 \begin{align}\label{energy}\|v_\nu(t)\|_{L^2}^2+2\nu\int_0^t\|\nabla v_\nu\|_{L^2}^2\dif s\leq \|v_{\text{in}}\|_{L^2}^2+2\int_0^t\langle f_\nu,v_\nu\rangle\dif s.
 \end{align}
\ed

The main aim of this section is to prove the $L^{p}$-in-time version of the Kolmogorov $4/5$ law for the forced Navier--Stokes equations \eqref{eq:v}
under the assumption (\textbf{H}).   By interpolation we obtain that for any $p>1$ there exists a $\beta(p)>0$ such that
\begin{align}\label{bd:p}
	\sup_{\nu\in(0,1)}\|v_\nu\|_{L^p(0,1;H^{\beta(p)})}<\infty.\end{align}

In general it is not easy to construct  Leray--Hopf solutions satisfying assumption (\textbf{H}). Recently, such a result was obtained  in \cite{CCS22, BCCDLS22}. In particular, smooth solutions to the forced Navier--Stokes equations were constructed  such that Assumption (\textbf{H})  holds true.
More precisely,  the following result was shown in \cite[Theorem A]{BCCDLS22}.

\bt\label{th:bccds}
For any $\alpha<1/3$ there exist $\sigma>0$, a divergence-free initial data $v_{\text{in}}\in C^\infty(\mT^3;\mR^3)$ with $\int_{\mT^3}v_{\text{in}}=0$ and a family of forces $\{f_\nu\}_{\nu\in [0,1]}\subset C^\infty([0,1]\times \mT^3;\mR^3)$ satisfying \eqref{eq:v} such that
 for each $\nu\in(0,1)$ there exists a unique Leray--Hopf  solution to \eqref{eq:v} with $v^\nu(0)=v_{\text{in}}$ satisfying
 \begin{align}\label{est:alpha}
 	\sup_{\nu\in (0,1)}\Big(\|v_\nu\|_{L^3(0,1;C^\alpha)}+\|v_\nu\|_{L^\infty(0,1;L^\infty)}\Big)<\infty,\end{align}
 and  the following anomalous dissipation holds
 $$\limsup_{\nu\to0}\nu\int_0^1\int_{\mT^3}|\nabla v_\nu|^2\dif x\dif t>\frac14.$$
In particular, \eqref{energy} holds for the solution with inequality replaced by equality. Furthermore, $f_\nu\to f_0$ in $L^{1+\sigma}(0,1;C^\sigma(\mT^3))$ and $v_\nu\to v_0$ in $L^2(0,1;L^2)$ as $\nu\to0$ and $(v_0, f_0)$ satisfies \eqref{eq:v} with $\nu=0$.
\et

Let us now introduce further necessary notations and tools in order to formulate and establish the $L^p$-in-time version of Kolmogorov's $4/5$ and $4/3$ laws.
Define
\begin{align*}
	\Gamma_\nu(t,h):=\int_{\mT^3}v_\nu(t,x)\otimes v_\nu(t,x+h)\dif x,\quad t\in [0,1], h\in\mR^3,
\end{align*}
and for each $k=1,2,3$  the third order structure matrix
\begin{align*}
	D^k_\nu(t,h)=\int_{\mT^3}(\delta_h v_\nu(t,x)\otimes \delta_h v_\nu(t,x))\delta_hv_\nu^k(t,x)\dif x,\quad t\in [0,1], h\in\mR^3,
\end{align*}
where \begin{align*}
	\delta_h v_\nu(x)=v_\nu(x+h)-v_\nu(x),\quad h\in\mR^3.
\end{align*}
Since $v_\nu$ is a Leray-Hopf solution, which stays in $L^3(0,1;L^3)$, $D_\nu^k$ is finite for each fixed $\nu>0$.
As the next step, we derive the K\'arm\'an--Howarth--Monin (KHM) relation  as in \cite[Proposition 3.1]{BCZPSW19}.

\bp\label{prop:31}
Suppose that $v_\nu$, $\nu\in(0,1)$, are Leray--Hopf solutions to the forced Navier--Stokes equations \eqref{eq:v} on $[0,1]\times \mT^3$ with forces $f_\nu\in L^1(0,1;L^2)$, $\nu\in(0,1)$. Let $\eta(h)=(\eta_{ij}(h))_{ij=1}^3$ be a smooth, compactly supported function of the form
\begin{align*}
	\eta(h)=\phi(|h|)I+\varphi(|h|)\hat{h}\otimes \hat{h}, \quad \hat h=\frac{h}{|h|},
\end{align*}
where $\phi$ and $\varphi$ are smooth and compactly supported on $(0,\infty)$. Then it holds that for any $t\in [0,1]$
\begin{align*}
	&\int_{\mR^3}\eta(h):\Gamma_\nu(t,h)\dif h-\int_{\mR^3}\eta(h):\Gamma_\nu(0,h)\dif h\\=&-\frac12\sum_{k=1}^3\int_0^t\int_{\mR^3}\p_k\eta(h):D^k_\nu(r,h)\dif h\dif r+2\nu\int_0^t\int_{\mR^3}\Delta \eta(h):\Gamma_\nu(r,h)\dif h\dif r\\&+ 2\int_0^t\int_{\mR^3}\int_{\mT^3}\eta(h):f_\nu\otimes T_hv_\nu\dif x\dif h\dif r,
\end{align*}
where $T_h v_\nu(t,x)=v_\nu(t,x+h)$.
\ep
\begin{proof}
	The proof follows from the same argument as in \cite[Proposition 3.1]{BCZPSW19}. The only difference is that we do not have the martingale part but have the  extra force $f_\nu$. The idea in \cite[Proposition 3.1]{BCZPSW19} is to consider the evolution of $\int_{\mT^3} v_\nu\otimes T_hv_\nu \dif x$, which gives us the additional  term
	$$\int_{\mT^3}[f_\nu\otimes T_hv_\nu+v_\nu\otimes T_hf_\nu]\dif x.$$
	Testing by the test function $\eta$ we therefore obtain	the desired result by symmetry of $\eta$.
\end{proof}

Now, we prove the $L^p$-in-time version of the $4/3$ law and the $4/5$ law. To this end,  we recall the averaged structure functions from \eqref{eq:s1} and \eqref{eq:s2} for any $\ell>0$ and $t\in[0,1]$
\begin{align*}
	S_0^\nu(t,\ell)&=\frac1{4\pi}\int_{\mS^2}\int_{\mT^3}|\delta_{\ell\hat{n}} v_\nu(t)|^2\delta_{\ell\hat{n}} v_\nu(t)\cdot \hat{n}\dif x\dif S(\hat n),\\
	S_{\|}^\nu(t,\ell)&=\frac1{4\pi}\int_{\mS^2}\int_{\mT^3}(\delta_{\ell\hat{n}} v_\nu(t)\cdot \hat{n})^3\dif x\dif S(\hat{n}).
\end{align*}
Since $v_\nu$ is a Leray-Hopf solution, which stays in $L^3(0,1;L^3)$, $S_0^\nu$ and $S_{\|}^\nu$ are finite for each fixed $\nu>0$.
We define further the spherical average
\begin{align*}
	\bar{\Gamma}_\nu(t,\ell):=\frac1{4\pi} \int_{\mS^2} I:\Gamma_\nu(t,\ell\hat n)\dif S(\hat n),
\end{align*}
and {the energy dissipation rate} $$\eps_\nu(t):=\frac12\|v_\nu(0)\|_{L^2}^2-\frac12\|v_\nu(t)\|_{L^2}^2+\int_0^t\<f_\nu,v_\nu\>\dif r.$$
Note that the energy inequality implies that
$$\eps_\nu(t)\geq \int_0^t\nu\|\nabla v_\nu\|_{L^2}^2\dif r.$$
The equality holds if the energy equality holds. In particular for the solutions in Theorem \ref{th:bccds} we have
$$\limsup_{\nu\to0}\eps_\nu(1)=\limsup_{\nu\to0}\int_0^1\nu\|\nabla v_\nu\|_{L^2}^2\dif r>\frac14,$$
by anomalous dissipation at time $1$.

We have all in hand to  state and prove our first main results.

\bt\label{th:34}
 Let  $v_\nu$, $\nu\in(0,1)$,  be Leray--Hopf solutions to the forced Navier--Stokes equations \eqref{eq:v} with forces $f_\nu$,  $\nu\in(0,1)$, the same initial datum $v_{\text{in}}\in H^\beta$ for some $\beta>0$, and satisfying the assumption {\em(\textbf{H})}. Then  there exists $\ell_D=\nu^K$ with $K<\frac1{2-\alpha q}$  such that for any $p\in[1,\infty)$
\begin{align*}
	\lim_{\ell_I\to0}\limsup_{\nu\to0}\sup_{\ell\in [\ell_D,\ell_I]}\left\|\int_0^\cdot\frac{S_0^\nu(r,\ell)}{\ell}\dif r+\frac43\eps_\nu(\cdot)\right\|_{L^p(0,1)}=0.
\end{align*}
Here  $\eps_\nu(t)=\int_0^t\nu\|\nabla v_\nu\|_{L^2}^2\dif r$ if the energy equality holds.

\et
\begin{proof} It is sufficient to prove
	\begin{align}\label{4/3law}
		\lim_{\ell_I\to0}\limsup_{\nu\to0}\sup_{\ell\in [\ell_D,\ell_I]}\sup_{\|\psi\|_{L^{p'}(0,1)}\leq1}\left|\int_0^1\psi(t)\left[\int_0^t\frac{S_0^\nu(r,\ell)}{\ell}\dif r+\frac43\eps_\nu(t)\right]\dif t\right|=0,
	\end{align}
	 for $\psi\in L^{p'}(0,1)$ with $p'=\frac{p}{p-1}$ and the result follows from duality.
	Choosing $\eta(h)=\phi(|h|)I$ with $\phi$ smooth and compactly supported on $(0,\infty)$ in Proposition \ref{prop:31},  we  obtain on $[0,t]$ for $t\leq1$
\begin{equation}\label{eqkmh2}
	\aligned
	&\int_{\mR^3} \phi(|h|)[\<v_\nu,T_hv_\nu\>(t)-\<v_\nu,T_hv_\nu\>(0)]\dif h+\frac12\sum_{k=1}^3\int_0^t\int_{\mR^3}\p_k\eta(h):D^k_\nu(r,h)\dif h\dif r
	\\&=2\nu\int_0^t\int_{\mR^3}\phi(|h|) \Delta\<v_\nu,T_hv_\nu\>\dif h\dif r+2\int_0^t\int_{\mR^3}\phi(|h|)\<f_\nu,T_hv_\nu\>\dif h\dif r.
\endaligned
\end{equation}
We  define the following spherical averages
$$
J_\nu(t,\ell):=\int_{\mathbb{S}^{2}} \<v_\nu,T_{\ell \hat n}v_\nu\>(t)\dif S(\hat n), \quad t\in [0,1],\ell\geq0,
$$
\begin{align*}
	\bar{f}_\nu(t,\ell):=\frac1{4\pi}\int_{\mS^2}\int_{\mT^3}\< f_\nu, T_{\ell \hat n}v_\nu\>\dif x\dif S(\hat n),\quad t\in [0,1], \ell\geq0.
\end{align*}
By \cite[Lemma 2.8]{BCZPSW19} we know that $\ell\to\int_0^t \bar{\Gamma}_\nu(r,\ell)\dif r$ is in $C^2$. Hence, we could write \eqref{eqkmh2} as
\begin{align*}
		&\int_{\mR^+}\ell^2\phi(\ell)\left(\int_0^t\Big(\nu \partial^2_\ell\bar\Gamma_\nu(r,\ell)+\nu \frac2\ell\partial_\ell\bar\Gamma_\nu(r,\ell)+\bar f_\nu(r,\ell)\Big)\dif r+\frac1{8\pi}(J_\nu(0,\ell)-J_\nu(t,\ell))\right)\dif \ell
		\\&=\frac14\int_{\mR^+}\int_0^tS_0^\nu(r,\ell)\ell^2\dif r\phi'(\ell)\dif \ell.
	\end{align*}
	Using integration by parts we find
	\begin{align*}
		& \p_\ell\Big(\ell^3\int_0^t\frac{S_0^\nu(r,\ell)}{\ell}\dif r\Big)
		\\&=-\ell^2\left(\int_0^t(4\nu \partial^2_\ell\bar\Gamma_\nu(r,\ell)+4\nu \frac2\ell\partial_\ell\bar\Gamma_\nu(r,\ell)+4\bar f_\nu(r,\ell))\dif r+\frac1{2\pi}(J_\nu(0,\ell)-J_\nu(t,\ell))\right).
	\end{align*}
	Taking integration with respect to $\ell$ we get since the boundary term vanishes
	\begin{align*}
		&\int_0^t\frac{S_0^\nu(r,\ell)}{\ell}\dif r
		\\&=-\frac1{\ell^3}\int_0^\ell\tau^2\left(\int_0^t\Big(4\nu\partial_\ell^2 \bar\Gamma_\nu(r,\tau)+4\nu \frac2\tau\partial_\ell\bar\Gamma_\nu(r,\tau)+4\bar f_\nu(r,\tau)\Big)\dif r+\frac1{2\pi}(J_\nu(0,\tau)-J_\nu(t,\tau))\right)\dif \tau.
	\end{align*}
	Moreover, a direct calculation yields
	\begin{align*}
		\frac1{\ell^3}\int_0^t\int_0^\ell\tau^2\Big(4\nu \partial^2_\tau\bar\Gamma_\nu(r,\tau)+4\nu \frac2\tau\partial_\tau\bar\Gamma_\nu(r,\tau)\Big)\dif \tau\dif r=\int_0^t\frac{4\nu\partial_\ell\bar\Gamma_\nu(r,\ell)}{\ell}\dif r.
	\end{align*}
	The right hand side can be estimated as
  in \cite{GK23}:
			\begin{align}\label{eq:p}
			&\left|\int_0^t\frac{4\nu\partial_\ell\bar\Gamma_\nu(r,\ell)}{\ell}\dif r\right|
			\\=&\left|	\int_0^t\frac{\nu}{\pi\ell}\int \sum_{i,j} (v_i(t,x)-v_i(t,x+\ell\hat n))\partial_j v_i(t,x+\ell\hat{n}) \hat n_j\dif x\dif S(\hat n) \dif r\right|\no
			\\&\lesssim \frac{\nu}\ell \int_0^t\|\nabla v_\nu\|_{L^2}\| \delta_{\ell\hat n} v_\nu\|_{L^2} \dif r\no
			\lesssim \frac{\nu^{1/2}}\ell\left( \int_0^t\| \delta_{\ell\hat n} v_\nu\|_{L^2}^2\right)^{1/2}\left(\nu \int_0^t\|\nabla v_\nu\|_{L^2}^2\right)^{1/2}.\no
		\end{align}
	Here, we used $\int v_i(t,x+\ell\hat n))\partial_j v_i(t,x+\ell\hat{n})\dif x=0$. Then, by the regularity of $v_\nu\in L^2(0,1;H^{\alpha q/2})$ the above is bounded by
	$$C\nu^{1/2}\ell^{\alpha q/2-1},$$
	which converges to zero under our condition on $\ell_D$.
Since the solutions are Leray--Hopf and satisfy Assumption (\textbf{H}) it holds
\begin{align*}
	\frac12\sup_{\nu\in (0,1)}\|v_\nu(t)\|_{L^2}^2+\sup_{\nu\in (0,1)}\nu \int_0^t\|\nabla v_\nu\|_{L^2}^2\leq \|v_{\text{in}}\|_{L^2}^2+2\sup_{\nu\in(0,1)}\int_0^t\<f_\nu,v_\nu\>\dif r<\infty,
\end{align*}
which combined with Assumption  (\textbf{H}) implies that for $\ell_D\geq \nu^{1/(2-\alpha q)-\kappa}$ with $\kappa>0$
$$\lim_{{\ell_I}\to0}\limsup_{\nu\to0}\sup_{\ell\in[\ell_D,\ell_I]}\sup_{t\in[0,1]}\left|\int_0^t\frac{4\nu\partial_\ell\bar\Gamma_\nu(r,\ell)}{\ell}\dif r\right|=0.$$
Hence we arrive at for $t\in [0,1]$
\begin{align}\label{eq:S0}
	\int_0^t\frac{S^{\nu}_0(r,\ell)}{\ell}dr=-\frac1{\ell^3}\int_0^\ell\tau^2\int_0^t4\bar f_\nu(r,\tau)\dif r\dif \tau+\frac1{2\pi\ell^3}\int_0^\ell \tau^2[J_\nu(t,\tau)-J_\nu(0,\tau)]\dif\tau+O(\nu^{\kappa}).
\end{align}
Now, we multiply both sides by a  function $\psi\in L^{p'}(0,1)$ and  integrate over $t$ to obtain
\begin{equation}\label{eq:S01}
	\aligned
	\int_0^1\psi(t)\int_0^t\frac{S^{\nu}_0(r,\ell)}{\ell}\dif r\dif t=&-\frac1{\ell^3}\int_0^\ell\tau^2\int_0^1\int_0^t4\bar f_\nu(r,\tau)\dif r \psi(t)\dif t\dif \tau\\&+\frac1{2\pi\ell^3}\int_0^\ell \tau^2\int_0^1[J_\nu(t,\tau)-J_\nu(0,\tau)]\psi(t)\dif\tau \dif t+O(\nu^{\kappa}).
	\endaligned
\end{equation}
In the following we want to replace $\bar f_\nu(r,\tau), J_\nu(t,\tau)$ and $J_\nu(0,\tau)$ by $\bar f_\nu(r,0), J_\nu(t,0)$ and $J_\nu(0,0)$, respectively.
By the regularity of the initial data we know that
$$\lim_{\ell\to0}\frac1{2\pi\ell^3}\int_0^\ell \tau^2[J_\nu(0,\tau)-J_\nu(0,0)]\dif\tau=0,$$ as well as
$$\frac1{2\pi\ell^3}\int_0^\ell \tau^2J_\nu(0,0)\dif\tau =\frac23\|v_{\text{in}}\|_{L^2}^2.$$
Moreover, for $J_\nu$ at time $t$ we use Assumption (\textbf{H}) and \eqref{bd:p} to have
\begin{equation}\label{*}
	\aligned
	&
	\left|\frac1{2\pi\ell^3}\int_0^1\int_0^\ell \tau^2[J_\nu(t,\tau)-J_\nu(t,0)]\dif \tau \psi(t) \dif t\right|\\
&\lesssim\frac1{\ell^3}\int_0^\ell \tau^2\int_{\mathbb{S}^{2}}\int_0^1 \|v_\nu(t)\|_{L^2}\|T_{\tau \hat n}v_\nu(t)-v_\nu(t)\|_{L^2}|\psi(t)|\dif t\dif S(\hat n)\dif\tau\\
&\lesssim\frac1{\ell^3}\int_0^\ell \tau^{2+\beta(p)}\int_0^1 \|v_\nu(t)\|_{L^2}\|v_\nu(t)\|_{H^{\beta(p)}}|\psi(t)|\dif t\dif \tau\\
&\lesssim\frac1{\ell^3}\int_0^\ell \tau^{2+\beta(p)}\Big(\int_0^1 \|v_\nu(t)\|^p_{H^{\beta(p)}}\dif t\Big)^{1/p}\Big(\int_0^1|\psi(t)|^{p'}\dif t\Big)^{1/p'}\dif\tau\rightarrow0, \quad \ell\to0.
\endaligned
\end{equation}
Here, the  convergence holds uniformly in $\nu$. Hence, we could replace $J_\nu(t,\tau)$ by $J_\nu(0,\tau)$ in \eqref{eq:S01}.

 In addition, using uniform in $\nu$ Sobolev regularity of $v_\nu$ from Assumption (\textbf{H}) we  have for $p_0$ satisfying $\frac1{p_0}+\frac1{1+\sigma}=1$ and $\beta_0=\beta(p_0)$ from \eqref{bd:p}
	\begin{align*}
		\sup_{\nu\in (0,1)}\int_0^t|\bar{f}_\nu(r,{\tau})-\bar{f}_\nu(r,0)|\dif r&\lesssim {\tau}^{\beta_0} \sup_{\nu\in (0,1)}\int_0^t\|f_\nu\|_{L^2}\|v_\nu\|_{H^{\beta_0}}\dif r
\\&\lesssim {\tau}^{\beta_0}\left(\int_0^t\|f_\nu\|_{L^2}^{1+\sigma}\dif r\right)^{1/(1+\sigma)}\left(\int_0^t\|v_\nu\|^{p_0}_{H^{\beta_0}}\dif r\right)^{1/p_0}\lesssim  {\tau}^{\beta_0},
	\end{align*}
which implies that
	\begin{align*}
		\frac4{\ell^3}\int_0^t\int_0^\ell\tau^2\bar f_\nu(r,\tau)\dif \tau\dif r&=\frac43\int_0^t\bar f_\nu(r,0)\dif r+\frac4{\ell^3}\int_0^\ell\tau^2\int_0^t(\bar f_\nu(r,\tau)-\bar f_\nu(r,0))\dif r\dif \tau
\\&\to \frac43\int_0^t\bar f_\nu(r,0)\dif r=\frac43\int_0^t\<f_\nu,v_\nu\>\dif r,\quad \ell\to 0.
	\end{align*}
	Also here the  convergence holds uniformly in $\nu$.

Combining the above calculations we finally deduce
\begin{align*}
	\lim_{{\ell_I}\to0}\limsup_{\nu\to0}\sup_{\ell\in[\ell_D,\ell_I]}\sup_{\|\psi\|_{L^{p'}(0,1)}\leq1}\bigg[&\int_0^1\psi(t)\bigg(\int_0^t\frac{S^{\nu}_0(r,\ell)}{\ell}\dif r\\&-\frac23\bigg[\|v_\nu(t)\|_{L^2}^2-\|v_\nu(0)\|_{L^2}^2-2\int_0^t\<f_\nu,v_\nu\>\dif r\bigg]\bigg)\dif t\bigg]=0,
\end{align*}
which implies \eqref{4/3law}.
\end{proof}

\bt\label{45}
Let $v_\nu$, $\nu\in (0,1)$, be Leray--Hopf solutions to the forced Navier--Stokes equations \eqref{eq:v} with forces $f_\nu$, $\nu\in (0,1)$, the same initial datum $v_{\text{in}}\in H^\beta $ for some $\beta>0$. Then  there exists $\ell_D=\nu^K$ with $K<\frac1{2-\alpha q}$ such that  for any $p\in[1,\infty)$
\begin{align*}
	\lim_{\ell_I\to0}\limsup_{\nu\to0}\sup_{\ell\in [\ell_D,\ell_I]}\left\|\int_0^\cdot\frac{S_{\|}^\nu(r,\ell)}{\ell}\dif r+\frac45\eps_\nu(\cdot)\right\|_{L^p(0,1)}=0.
\end{align*}
Here $\eps_\nu(t)=\int_0^t\nu\|\nabla v_\nu\|_{L^2}^2\dif r$ if the energy equality holds.
\et

\begin{proof} It suffices to prove that
	\begin{align*}
		\lim_{\ell_I\to0}\limsup_{\nu\to0}\sup_{\ell\in [\ell_D,\ell_I]}\sup_{\|\psi\|_{L^{p'}(0,1)}\leq1}
		\left|\int_0^1\psi(t)\left[\int_0^t\frac{S_{\|}^\nu(r,\ell)}{\ell}\dif r+\frac45\eps_\nu(t)\right]\dif t\right|=0,
	\end{align*}
for $p'=\frac{p}{p-1}$.
	Choosing $\eta(h)=\varphi(|h|)\hat h\otimes \hat h, \hat h=\frac{h}{|h|}$ in Proposition \ref{prop:31}, we obtain on $[0,t]$ with $t\in[0,1]$
	\begin{equation}\label{eq:bl}
		\aligned
		&\int_{\mR^3}\int_{\mT^3} (\varphi(|h|)\hat h\otimes \hat h):[(v_\nu\otimes T_hv_\nu)(t)-(v_\nu\otimes T_hv_\nu)(0)]\dif x\dif h
		\\&\quad+\sum_{k=1}^3\int_0^t\int_{\mR^3}\p_{h_k}(\varphi(|h|)\hat h\otimes \hat h):\Big(\frac12 D^k_\nu(r,h)+2\nu\p_{h_k}\Gamma_\nu(r,h)\Big)\dif h\dif r
		\\&=2\int_0^t\int_{\mR^3}\int_{\mT^3}\varphi(|h|)\hat h\otimes \hat h:f_\nu\otimes T_h v_\nu\dif x\dif h\dif r.
		\endaligned
	\end{equation}
Let
\begin{align*}&
G_\nu(t,\ell):=\int_{\mT^3}\int_{\mathbb{S}^{2}} \hat n\otimes \hat n:[v_\nu\otimes T_{\ell \hat n}v_\nu(t)]\dif S(\hat n)\dif x.
\end{align*}
The first term in \eqref{eq:bl} is just
\begin{align*}
\int_{\mR^{+}}\varphi(\ell)\ell^2	(G_\nu(t,\ell)-G_\nu(0,\ell))\dif \ell.
\end{align*}
	Using the same calculation as in \cite[Section 5]{BCZPSW19} we could write the second term of the left hand side  of \eqref{eq:bl} as
	\begin{align*}
		2\pi \int_0^t\int_{\mR^+}\Big(\ell^2\varphi'(\ell)-2\ell \varphi(\ell)\Big)(S_{\|}^\nu(r,\ell)+4\nu H_\nu(r,\ell))\dif \ell\dif r+4\pi\int_0^t\int_{\mR^+}\ell\varphi(\ell)S_0^\nu(r,\ell)\dif \ell\dif r,
	\end{align*}
	where
	\begin{align*}
		H_\nu(t,\ell):=\frac1{4\pi} \int_{\mS^2}\int_{\mT^3}(\hat n\cdot v_\nu)(\hat n\otimes \hat n:T_{\hat n\ell}\nabla v_\nu)(t)\dif x\dif S(\hat n).
	\end{align*}
	We consider the right hand side of \eqref{eq:bl} and write it as
	\begin{align*}
		8\pi\int_0^t\int_{\mR^+}\ell^2\varphi(\ell)\widetilde f_\nu(r,\ell)\dif \ell,
	\end{align*}
	with
	\begin{align*}
		\widetilde f_\nu(t,\ell):=	\frac1{4\pi}\int_{\mS^2}\int_{\mT^3}\hat n\otimes \hat n:[f_\nu\otimes T_{\ell \hat n} v_\nu](t)\dif x\dif S(\hat n).
	\end{align*}
	Substituting these calculations into \eqref{eq:bl} and using the fact that $\ell^2\varphi'(\ell)-2\ell\varphi(\ell)=\ell^4(\ell^{-2}\varphi(\ell))'$, we arrive at
	\begin{align*}
		&\p_\ell\left[\int_0^t\ell^4(S_{\|}^\nu(r,\ell)+4\nu H_\nu(r,\ell))\dif r\right]-2\ell^3\int_0^tS_0^\nu(r,\ell)\dif r+4\ell^4\int_0^t\widetilde{f}_\nu(r,\ell)\dif r\\
		&\qquad=\frac{\ell^4}{2\pi}(G_\nu(t,\ell)-G_\nu(0,\ell)).
	\end{align*}
	Integrating with respect to $\ell$ on both sides leads to
	\begin{equation}\label{eq:S}
		\aligned
		&\int_0^t\ell^4\big(S_{\|}^\nu(r,\ell)+4\nu H_\nu(r,\ell)\big)\dif r\\&=2\int_0^t\int_0^\ell\tau^3S_0^\nu(r,\tau)\dif \tau\dif r-4\int_0^t\int_0^\ell\tau^4\widetilde{f}_\nu(r,\tau)\dif \tau\dif r+\frac1{2\pi}\int_0^\ell \tau^4(G_\nu(t,\tau)-G_\nu(0,\tau))\dif \tau,
		\endaligned
	\end{equation}
	where the  boundary term vanishes since $v_\nu$ satisfies energy inequality.

	Dividing both sides of \eqref{eq:S} by $\ell^5$ we obtain
	\begin{align}\label{eq:S1}
		\int_0^t\frac{S_{\|}^\nu(r,\ell)}{\ell}\dif r=&-4\nu\int_0^t\frac{ H_\nu(r,\ell)}{\ell}\dif r+2\ell^{-5}\int_0^t\int_0^\ell\tau^3S_0^\nu(r,\tau)\dif \tau\dif r
		\\&-4\ell^{-5}\int_0^t\int_0^\ell\tau^4\widetilde{f}_\nu(r,\tau)\dif \tau\dif r+\frac1{2\pi\ell^5}\int_0^\ell \tau^4(G_\nu(t,\tau)-G_\nu(0,\tau))\dif \tau.\no
	\end{align}
In the following we consider each term on the right hand side of \eqref{eq:S1}.	By definition of $H_\nu$ we have
		similarly as above
		\begin{align*}
			&\left|\int_0^t\frac{\nu H_\nu(r,\ell)}{\ell}\dif r\right|
			\\&=\left|	\int_0^t\frac{\nu}{4\pi\ell}\int \sum_{j} (v\cdot \hat n(t,x)-v\cdot \hat n(t,x+\ell\hat n))\partial_j (v \cdot \hat n)(t,x+\ell\hat{n}) \hat n_j\dif x\dif S(\hat n) \dif r\right|
			\\&\lesssim \frac{\nu}\ell \int_0^t\|\nabla v_\nu\|_{L^2}\| \delta_{\ell\hat n} v_\nu\|_{L^2} \dif r
			\lesssim \frac{\nu^{1/2}}\ell\left( \int_0^t\| \delta_{\ell\hat n} v_\nu\|_{L^2}^2\right)^{1/2}\left(\nu \int_0^t\|\nabla v_\nu\|_{L^2}^2\right)^{1/2},
		\end{align*}
		where we used $\int v\cdot \hat n(t,x+\ell\hat n))\partial_j (v\cdot \hat n)(t,x+\ell\hat{n})\dif x=0.$ By the regularity of $v_\nu\in L^2(0,1;H^{\alpha q/2})$, the above is bounded by
		$$C\nu^{1/2}\ell^{\alpha q/2-1},$$
		which converges to zero under our  condition on $\ell_D$.
	Using Assumption (\textbf{H}) we  have for the same  $p_0,\beta_0$ as in the proof of Theorem \ref{th:34}
\begin{align*}
		\sup_{\nu\in (0,1)}\int_0^t|\widetilde{f}_\nu(r,{\tau})-\widetilde{f}_\nu(r,0)|\dif r\lesssim {\tau}^{\beta_0} \sup_{\nu\in (0,1)}\int_0^t\|v_\nu\|_{H^{\beta_0}}\|f_\nu\|_{L^2}\dif r\lesssim {\tau}^{\beta_0} .
	\end{align*}
	Hence, using
	$$\widetilde{f}_\nu(r,0)=\frac13\<f_\nu,v_\nu\>(r),$$ we obtain
	\begin{align*}
	&	4\ell^{-5}\int_0^t\int_0^\ell\tau^4\widetilde{f}_\nu({r,\tau})\dif \tau\dif r
		\\&=\frac4{15}\int_0^t\<f_\nu,v_\nu\>(r)\dif r+4\ell^{-5}\int_0^t\int_0^{\ell}\tau^4(\widetilde{f}_\nu(r,\tau)-\widetilde{f}_\nu(r,0))\dif \tau\dif r
		=\frac4{15}\int_0^t\<f_\nu,v_\nu\>\dif r+{O}(\ell^{{\beta_0}}).
	\end{align*}
By the regularity of the initial data, we obtain
	\begin{align*}&\lim_{\ell\to0}\limsup_{\nu\to0}\sup_{\ell\in[\ell_D,\ell_I]}\frac1{\ell^5}\int_0^\ell \tau^4[G_{{\nu}}(0,\tau)-G_{{\nu}}(0,0)]\dif\tau =0.
\end{align*}
By the same argument as in the calculation in \eqref{*} we obtain
\begin{align*}&\lim_{\ell\to0}\limsup_{\nu\to0}\sup_{\ell\in[\ell_D,\ell_I]}\sup_{\|\psi\|_{L^{p'}(0,1)}\leq1}\frac1{\ell^5}\int_0^1 \psi(t)\int_0^\ell \tau^4[G_{{\nu}}(t,\tau)-G_{{\nu}}(t,0)]\dif\tau \dif t=0.
\end{align*}
Hence, using $G_{{\nu}}(t,0)=\frac{4\pi}3\|v_\nu(t)\|_{L^2}^2$ and testing with a function $\psi\in L^{p'}$ as before we obtain
	\begin{align}\label{eq:S10}
	\int_0^1\psi(t)\int_0^t\frac{S_{\|}^\nu(r,\ell)}{\ell}\dif r\dif t&=2\ell^{-5}\int_0^1\psi(t)\int_0^t\int_0^\ell\tau^3S_0^\nu(r,\tau)\dif \tau\dif r\dif t-\frac{4}{15}\int_0^1\psi(t)\int_0^t\<f_\nu,v_\nu\>\dif r\dif t
	\\&\quad+\frac{2}{15}\int_0^1\psi(t)(\|v_\nu(t)\|_{L^2}^2-\|v_\nu(0)\|^2_{L^2})\dif t+{O}(\nu^{
	{\kappa}})+{O}(\ell^{{\beta_0}}).\no
\end{align}
Using \eqref{eq:p} and \eqref{eq:S0} in the proof of Theorem \ref{th:34} we have
	\begin{align*}
		&	\lim_{\ell_I\to0}\limsup_{\nu\to0}\sup_{\ell\in [\ell_D,\ell_I]}\sup_{\|\psi\|_{L^{p'}(0,1)}\leq1}\left|\int_0^1\psi(t)\left[\int_0^t2\ell^{-5}\int_0^\ell \tau^3 S_0^\nu(r,\tau)\dif \tau\dif r+\frac45\times \frac23\varepsilon_\nu(t)\right]\dif t\right|=0
		.\end{align*}
	Combining the above calculations the  result follows.
\end{proof}

Combining Theorem~\ref{th:34}, Theorem~\ref{45} and Theorem \ref{th:bccds} we obtain the following corollary.

\bc\label{cor:2.6}
 Suppose that $v_\nu$, $\nu\in (0,1)$, are the solutions to the forced Navier--Stokes equations \eqref{eq:v} from Theorem \ref{th:bccds}. Then there exists $\ell_D=\nu^K$ with $K<\frac1{2-\alpha q}$ and for any {$p\in [1,\infty)$}
	\begin{align*}
	\lim_{\ell_I\to0}\limsup_{\nu\to0}\sup_{\ell\in [\ell_D,\ell_I]}\left\|\int_0^\cdot\frac{S_{0}^\nu(r,\ell)}{\ell}\dif r+\frac43\nu\int_0^\cdot\|\nabla v_\nu\|_{L^2}^2\dif r \right\|_{L^p(0,1)}=0,
\end{align*}
and
	\begin{align*}
		\lim_{\ell_I\to0}\limsup_{\nu\to0}\sup_{\ell\in [\ell_D,\ell_I]}\left\|\int_0^\cdot\frac{S_{\|}^\nu(r,\ell)}{\ell}\dif r+\frac45\nu\int_0^\cdot\|\nabla v_\nu\|_{L^2}^2\dif r \right\|_{L^p(0,1)}=0.
	\end{align*}
\ec

To conclude this section, we discuss the proof of Theorem~\ref{th:1.2}.

\begin{proof}[Proof of  Theorem~\ref{th:1.2}]
The result follows from 
Theorem~\ref{th:34} and Theorem~\ref{45} and H\"older's inequality since  by definition
$$
\left\langle\left|\int_0^\mathfrak{t}\frac{S_{\|}^\nu(r,\ell)}{\ell}\dif r+\frac45\eps_\nu(\mathfrak{t})\right|^{p}\right\rangle = \int_{0}^{1}\psi(t)\left|\int_0^{t}\frac{S_{\|}^\nu(r,\ell)}{\ell}\dif r+\frac45\eps_\nu({t})\right|^{p}\dif t,
$$
where $\psi\in L^{1+\kappa}(0,1)$, for some $\kappa>0$, is the density of the law of the random time $\mathfrak{t}$ with respect to the Lebesgue measure and it holds $\|\psi\|_{L^{1+\kappa}(0,1)}\leq K$. A similar expression holds true in the case of the $4/3$ law.
\end{proof}

\section{Third order absolute structure function exponent}
\label{s:3}

In order to establish Theorem~\ref{thm:main2} via Theorem~\ref{thm:3.1} below, we first derive a modified version of the Kolmogorov $4/3$ law for the time average of the structure function $S^{\nu}_{0}$. Note that compared to the $L^{p}$-in-time version in Theorem~\ref{th:34}, here the rate of energy dissipation is replaced by a related but $\ell$-dependent  quantity. A similar result holds also true for the $4/5$ law as showed in Theorem~\ref{451}.

\bt\label{th:341}
Let  $v_\nu$, $\nu\in(0,1)$,  be Leray--Hopf solutions to the forced Navier--Stokes equations \eqref{eq:v} with forces $f_\nu$,  $\nu\in(0,1)$, the same initial datum $v_{\text{in}}\in H^\beta $ {for some $\beta>0$}
and satisfying the assumption {\em(\textbf{H})}. Then  there exists $\ell_D=\nu^K$ with $K<\frac1{2-\alpha q}$ such that
\begin{align}\label{eq:k}
	\lim_{\ell_I\to0}\limsup_{\nu\to0}\sup_{\ell\in [\ell_D,\ell_I]}\left|\int_0^1\frac{S_0^\nu(r,\ell)}{\ell}\dif r+\frac43\eps^\ell_\nu\right|=0,
\end{align}
where
\begin{align}\label{eq:k3}
\eps^\ell_\nu=\frac12\|v_\nu(0)\|_{L^2}^2-\frac{3}{8\pi\ell^3}\int_0^\ell\tau^2{\int_{\mathbb{S}^{2}}} \<v_\nu,T_{\tau\hat n}v_\nu\>(1)\dif S(\hat n) {\dif\tau}+\int_0^1\<f_\nu,v_\nu\>\dif r.
\end{align}
{Moreover, for the solutions from Theorem~\ref{th:bccds} it holds }
\begin{align}\label{eq:k2}
\limsup_{\nu\to0}\inf_{\ell\in (0,1)}\eps^\ell_\nu>\frac14.
\end{align}
\et

\begin{proof}
 The proof of \eqref{eq:k} follows from \eqref{eq:S0} by letting $t=1$ and keeping only  the term $$\frac1{2\pi\ell^3}\int_0^\ell\tau^2J_\nu(1,\tau)\dif \tau,
 $$
 while the other terms are handled by the same argument as in the proof of Theorem \ref{th:34}. In order to prove \eqref{eq:k2}, we note  that $$\<v_\nu,T_{\tau\hat n}v_\nu\>\leq \|v_\nu\|^2_{L^2},$$
which implies
\begin{align*}	
\eps^\ell_\nu
	&\geq\frac12\|v_\nu(0)\|_{L^2}^2-\frac{3}{8\pi\ell^3}\int_0^\ell\tau^2 {\int_{\mathbb{S}^{2}}}\|v_\nu(1)\|^2_{L^2}\dif S(\hat n){\dif\tau}+\int_0^1\<f_\nu,v_\nu\>\dif r
	\\&=\frac12\|v_\nu(0)\|_{L^2}^2- \frac12\|v_\nu(1)\|^2_{L^2}+\int_0^1\<f_\nu,v_\nu\>\dif r.
\end{align*}
Thus, for the solution obtained in Theorem \ref{th:bccds} it holds
$$
\limsup_{\nu\to0}\inf_{\ell\in (0,1)}\eps^\ell_\nu>\frac14.
$$
\end{proof}

{With Theorem~\ref{th:341} at hand, we are able to prove Theorem~\ref{thm:main2}. To this end, we recall that}
for any $\ell>0$  the third order absolute  structure function is given by
\begin{align*}
	S_3^\nu(t,\ell)&=\frac1{4\pi}\int_{\mS^2}\int_{\mT^3}|\delta_{\ell\hat{n}} v_\nu(t)|^3\dif x\dif S(\hat n).
\end{align*}
and the third order absolute structure function exponents read as
$$
\zeta_3:=\liminf_{\ell_I\to0}\inf_{\nu\in(0,1)}\inf_{\ell\in[\ell_D,\ell_I]}\frac{\log (\int_0^1S_3^{\nu}(r,\ell)\dif r)}{\log \ell},\qquad\bar\zeta_3:=\limsup_{\ell_I\to0}\liminf_{\nu\to0}\sup_{\ell\in [\ell_D,\ell_I]}\frac{\log (\int_0^1S_3^{\nu}(r,\ell)\dif r)}{\log \ell}.
$$
The result then reads as follows.

 \bt \label{thm:3.1}
 Suppose that $v_\nu$, $\nu\in(0,1)$, are the solutions to the forced Navier--Stokes equations \eqref{eq:v} obtained in Theorem~\ref{th:bccds} so that \eqref{est:alpha} holds for some $\alpha>0$.
 Then  $$3\alpha\leq{\zeta}_3\leq\bar{\zeta}_3\leq 1.
$$
\et

Here we note that Theorem \ref{th:bccds} holds for $\alpha<1/3$   arbitrary close to $1/3$.

\begin{proof}[Proof of Theorem \ref{thm:3.1}]
By Theorem \ref{th:341} we obtain
$$
\liminf_{\ell_I\to0}\limsup_{\nu\to0}\inf_{\ell\in [\ell_D,\ell_I]}\frac{|\int_0^{1} S_{0}^{\nu}(r,\ell)\dif r|}{\ell}>\frac18,
$$
Moreover, by
\begin{align*} \frac{\int_0^1 S_3^{\nu}(r,\ell)\dif r}{\ell} \geq\frac{|\int_0^{1} S_{0}^{\nu}(r,\ell)\dif r|}{\ell},
\end{align*}
we find that for any sequence $\ell_I^k\to0$ as $k\to\infty$  there exists $\nu_k\to0$ such that for any $\ell\in [\ell_D(\nu_k),\ell_I^k]$
\begin{align*} \int_0^1 S_3^{\nu}(r,\ell)\dif r \geq\frac18\ell.
\end{align*}
Taking logarithm on both sides and {dividing by} $\log \ell$,
we obtain $\bar\zeta_3\leq1$.

Finally, by the uniform in $\nu$ regularity of $v_{\nu}$ in $L^3(0,1;C^\alpha)$ we  obtain  for $\ell\in(0,1)$
\begin{align*} \sup_{\nu\in (0,1)}\int_0^1 S_3^\nu(r,\ell)\dif r\leq  C\ell^{3\alpha},
\end{align*}
which implies that $\zeta_3\geq3\alpha$. Hence  the result follows.
\end{proof}

Finally, we note that from the proof of Theorem \ref{45} we also obtain the following modified version of Kolmogorov $4/5$ law.

\bt\label{451}
Let  $v_\nu$, $\nu\in(0,1)$,  be Leray--Hopf solutions to the forced Navier--Stokes equations \eqref{eq:v} with forces $f_\nu$,  $\nu\in(0,1)$, the same initial datum $v_{\text{in}}\in H^\beta $ for some $\beta>0$
and satisfying the assumption {\em(\textbf{H})}. Then  there exists $\ell_D=\nu^K$ with $K<\frac1{2-\alpha q}$ such that
\begin{align*}
	\lim_{\ell_I\to0}\limsup_{\nu\to0}\sup_{\ell\in [\ell_D,\ell_I]}\left|\int_0^1\frac{S_{\|}^\nu(r,\ell)}{\ell}\dif r+\frac4{15}\bar\eps^\ell_\nu+
	\frac{8}{15}\tilde{\eps}_\nu^\ell\right|=0,
\end{align*}
where {$\tilde{\eps}_\nu^\ell=5\ell^{-5}\int_0^\ell\tau^4\eps_\nu^\tau\dif \tau$
with $\varepsilon^{\ell}_{\nu}$ was defined in \eqref{eq:k3} and } $$\bar\eps^\ell_\nu=\frac12\|v_\nu(0)\|_{L^2}^2-\frac{15}{8\pi\ell^5}\int_0^\ell\tau^4\int_{\mathbb{S}^{2}} \int_{\mT^3} \hat n\otimes \hat n: v_\nu\otimes T_{\tau\hat n}v_\nu(1)\dif x\dif S(\hat n)\dif\tau+\int_0^1\<f_\nu,v_\nu\>\dif r.$$
Moreover, for the solutions from Theorem~\ref{th:bccds} it holds
$$\limsup_{\nu\to0}\inf_{\ell\in (0,1)}\bar\eps^\ell_\nu>\frac18,\qquad \limsup_{\nu\to0}\inf_{\ell\in (0,1)}\tilde\eps^\ell_\nu>1/4.$$
\et
\begin{proof} The proof follows from \eqref{eq:S1} by taking $t=1$ and including the term  $\frac1{2\pi}\int_0^\ell\tau^4G_\nu(1,\tau)\dif \tau$ into $\bar\eps^\ell_\nu$.
The other terms follow the same argument as in the proof of Theorem \ref{45} and Theorem~\ref{th:341}.
Now, we prove the last result. Note that
$$\int_{\mathbb{S}^{2}}\int_{\mT^3} \hat n\otimes \hat n: v_\nu\otimes T_{\tau\hat n}v_\nu(1)\dif x\dif S(\hat n)\leq 4\pi\|v_\nu(1)\|^2_{L^2},$$
which implies that
\begin{align*}	\bar{\eps}^\ell_\nu
	&\geq\frac12\|v_\nu(0)\|_{L^2}^2- \frac32\|v_\nu(1)\|^2_{L^2}+\int_0^1\<f_\nu,v_\nu\>\dif r.
\end{align*}
Thus we obtain for the solutions obtained in Theorem \ref{th:bccds} uniformly in $\ell$
 $$\frac12\|v_\nu(0)\|_{L^2}^2- \frac32\|v_\nu(1)\|^2_{L^2}+\int_0^1\<f_\nu,v_\nu\>\dif r=\nu\int_0^1\|\nabla v_\nu\|_{L^2}^2-\|v_\nu(1)\|^2_{L^2}>\frac14-a_0^{\epsilon\delta/2},$$ where in the last step we use the notation $a_0^{\epsilon\delta/2}$ from \cite{CCS22} and the estimate on page 41-43 in \cite{CCS22} together with the fact that $u(1)=0$ in the notation of \cite{CCS22} to have $\|v_\nu(1)\|_{L^2}^2\leq a_0^{\epsilon\delta/2}$.
Then by choosing $a_0$ small enough we obtain
$$\limsup_{\nu\to0}\inf_{\ell\in (0,1)}\bar\eps^\ell_\nu>1/8.$$
\end{proof}

\end{document}